\documentclass[12pt]{amsart}

\usepackage{framed}
\usepackage{braket}
\usepackage{amsmath,amssymb,amsthm}
\usepackage{color}
\usepackage{bm}
\usepackage[all]{xy}
\usepackage{amscd}
\usepackage{graphicx}
\usepackage{ascmac}
\usepackage{BOONDOX-frak, mathrsfs}
\usepackage[top=30truemm,bottom=30truemm,left=25truemm,right=25truemm]{geometry}
\usepackage{mathtools}
\mathtoolsset{showonlyrefs=true}

\makeatletter
\@addtoreset{equation}{section}

\makeatother

\newcommand{\Z}{\mathbb{Z}}
\newcommand{\Q}{\mathbb{Q}}
\newcommand{\F}{\mathbb{F}}
\newcommand{\fA}{\mathfrak{A}}

\newcommand{\OO}{\mathcal{O}}
\newcommand{\II}{\mathcal{I}}
\newcommand{\JJ}{\mathcal{J}}
\newcommand{\CC}{\mathcal{C}}
\newcommand{\PP}{\mathcal{P}}
\newcommand{\cA}{\mathcal{A}}
\newcommand{\cB}{\mathcal{B}}

\newcommand{\SFitt}[1]{\Fitt^{[#1]}}
\newcommand{\SFittN}[1]{\Fitt^{\langle #1 \rangle}}
\newcommand{\wtil}[1]{\widetilde{#1}}
\newcommand{\ol}[1]{\overline{#1}}
\newcommand{\parenth}[1]{\left( #1 \right)}

\DeclareMathOperator{\Gal}{Gal}
\DeclareMathOperator{\Coker}{Cok}
\DeclareMathOperator{\Ker}{Ker}
\DeclareMathOperator{\Imag}{Im}
\DeclareMathOperator{\Fitt}{Fitt}

\DeclareMathOperator{\row}{row}
\DeclareMathOperator{\column}{column}
\DeclareMathOperator{\pd}{pd}
\DeclareMathOperator{\ram}{ram}
\DeclareMathOperator{\fin}{fin}
\DeclareMathOperator{\Cl}{Cl}
\DeclareMathOperator{\ord}{ord}
\DeclareMathOperator{\Hom}{Hom}
\DeclareMathOperator{\rank}{rank}

\let\oldenumerate\enumerate
\renewcommand{\enumerate}{
   \oldenumerate
   \setlength{\itemsep}{1pt}
   \setlength{\parskip}{0pt}
   \setlength{\parsep}{0pt}
}
\let\olditemize\itemize
\renewcommand{\itemize}{
   \olditemize
   \setlength{\itemsep}{1pt}
   \setlength{\parskip}{0pt}
   \setlength{\parsep}{0pt}
}

\theoremstyle{plain}
\newtheorem{thm}{Theorem}[section]
\newtheorem{lem}[thm]{Lemma}

\newtheorem{prop}[thm]{Proposition}
\newtheorem{cor}[thm]{Corollary}

\theoremstyle{definition}
\newtheorem{defn}[thm]{Definition}

\newtheorem{rem}[thm]{Remark}
\newtheorem{eg}[thm]{Example}

\title[Fitting ideals of class groups]
{Fitting ideals of class groups for CM abelian extensions}
\author[M.~Atsuta \& T. Kataoka]{Mahiro Atsuta and Takenori Kataoka}
\address{Faculty of Science and Technology, Keio University.
3-14-1 Hiyoshi, Kohoku-ku, Yokohama, Kanagawa 223-8522, Japan}
\email{atsuta0128@keio.ac.jp}
\email{tkataoka@math.keio.ac.jp}
\keywords{class groups, Fitting ideals, CM-fields, equivariant Tamagawa number conjecture}
\subjclass[2010]{11R29}
\date{\today}

\begin{document}

\begin{abstract}
Let $K/k$ be a finite abelian CM-extension and $T$ a suitable finite set of finite primes of $k$.
In this paper, we determine the Fitting ideal of the minus component 
of the $T$-ray class group of $K$, except for the $2$-component, assuming the validity of the equivariant Tamagawa number conjecture. 
As an application, we give a necessary and sufficient condition for
the Stickelberger element to lie in that Fitting ideal.
\end{abstract}

\maketitle


\section{Introduction}\label{Intro}

In number theory, 
the relationship between 
class groups and special values of $L$-functions is of great importance.
In this paper we discuss such a phenomenon for a finite abelian CM-extension $K/k$, that is, a finite abelian extension such that $k$ is a totally real field and $K$ is a CM-field.
We focus on the minus components of the (ray) class groups of $K$, except for the $2$-components, and study the Fitting ideals of them.

Let $\Cl_K$ denote the ideal class group of $K$.
Let $(-)^-$ denote the minus component after inverting the multiplication by $2$.
When $k = \Q$, 
Kurihara and Miura \cite{KM11} succeeded in proving a conjecture of Kurihara \cite{Kur03a} on a description of the Fitting ideal of $\Cl_K^-$ using the Stickelberger elements.
However, for a general totally real field $k$, the problem to determine the Fitting ideal of $\Cl_K^-$ is still open.

There seems to be an agreement that the {\it Pontryagin duals} (denoted by $(-)^{\vee}$) of the class groups are easier to deal with (see Greither-Kurihara \cite{GK08}).
In \cite{Gre07}, Greither determined the Fitting ideal of $\Cl_K^{\vee, -}$, assuming 
that the equivariant Tamagawa number conjecture (eTNC for short) holds
and that the group of roots of unity in $K$ is cohomologically trivial. 
Subsequently, Kurihara \cite{Kur20} generalized the results of Greither on $\Cl_K^{\vee, -}$ to results on $\Cl_K^{T, \vee, -}$, where $\Cl_K^T$ denotes the $T$-ray class group, for a finite set $T$ of finite primes of $k$.
This enables us, by taking suitably large $T$, to remove the assumption that the group of roots of unity is cohomologically trivial, 
though we still need to assume the validity of eTNC.
In recent work \cite{DK20}, 
Dasgupta and Kakde succeeded in proving {\it unconditionally} the same formula as Kurihara on the Fitting ideal of $\Cl_K^{T, \vee, -}$ (see \eqref{eq:dual_Fitt} below for the formula).

In this paper, for a general totally real field $k$,
we determine the Fitting ideal of $\Cl_K^{T, -}$ {\it without the Pontryagin dual}, assuming eTNC,
except for the $2$-component.
This problem has been considered to be harder than that on $\Cl_K^{T, \vee, -}$ and actually our result is more complicated.
Our main tool is the technique of shifts of Fitting ideals, which was established by the second author in \cite{Kata_05}. 

As an application of the description, we will obtain a necessary and sufficient condition for the Stickelberger element to be in the Fitting ideal of $\Cl_K^{T, -}$ (assuming eTNC).
Note that the question for the dualized version $\Cl_K^{T, \vee, -}$ is called the strong Brumer-Stark conjecture and is answered affirmatively by Dasgupta-Kakde \cite{DK20} unconditionally.

Though we mainly assume the validity of eTNC in this paper, we also obtain interesting unconditional results.
For instance, we will show that the Fitting ideal of $\Cl_K^{T, -}$ is always contained in that of $\Cl_K^{T, \vee, -}$, and that the inclusion is often proper.

In the rest of this section, we give precise statements of the main results. 

\subsection{Description of the Fitting ideal}\label{Intro1}

Let $K/k$ be a finite abelian CM-extension and put $G = \Gal (K/k)$.  
Let $S_{\infty}(k)$ be the set of archimedean places of $k$.
Let $S_{\ram} (K / k)$ be the set of places of $k$ which are ramified in $K/k$, including $S_{\infty}(k)$.
For each finite prime $v \in S_{\ram} (K / k)$,  
let $I_v \subset G$ denote the inertia group of $v$ in $G$ and $\varphi_v \in G / I_v$ the arithmetic Frobenius of $v$.
We then define elements $g_v$ and $h_v$ by
$$ g_v = 1-\varphi_v^{-1} + \# I_v \in \Z[G/I_v], \;\;\;
    h_v = 1 - \cfrac{\nu_{I_v}}{\# I_v}\varphi_v^{-1} + \nu_{I_v} \in \Q[G], $$
where we put $\nu_{I_v} = \sum_{\tau \in I_v} \tau$. 
These elements are introduced in \cite[Lemmas 6.1 and 8.3]{Gre07} and \cite[\S 2.2, equations (7) and (10)]{Kur20} (up to the involution).
Moreover, we define a $\Z[G]$-module $A_v$ by 
$$ A_v = \Z [G/I_v] / (g_v). $$

We write $\Z[G]^- = \Z[1/2][G] / (1+j)$, 
where $j$ is the complex conjugation in $G$. 
For any $\Z[G]$-module $M$, 
we also define the minus component by $M^-= M \otimes_{\Z[G]} \Z[G]^-$. 
Note that we are implicitly inverting the action of $2$.
For any $x \in M$, 
we write $x^-$ for the image of $x$ under the natural map $M \to M^-$.

In general, for a set $S$ of places of $k$, we write $S_K$ for the set of places of $K$ which lie above places in $S$.
We take and fix a finite set $T$ of finite primes of $k$ satisfying the following.
\begin{itemize}
\item
$T \cap S_{\ram} (K/k) = \emptyset$.
\item
$ K_T^\times = \{ x \in K^\times \mid \ord_w (x-1) > 0 \mbox{ for all primes } w \in T_K  \}$
is torsion free.
Here, $\ord_w$ denotes the normalized additive valuation.
\end{itemize}

Note that, if we fix an odd prime number $p$ and are concerned with the $p$-components, the last condition can be weakened to that $K_T^{\times}$ is $p$-torsion-free.
We consider the $T$-ray ideal class group of $K$ defined by 
\[
\Cl_K^T = \Coker \parenth{K_T^\times \overset{\oplus \ord_w}{\longrightarrow} 
\bigoplus_{w \notin T_K} \Z},
\] 
where $w$ runs over the finite primes of $K$ which are not in $T_K$.

For a character $\psi$ of $G$, 
we write $L (s,\psi)$ for the primitive $L$-function for $\psi$; 
this function omits exactly the Euler factors of primes dividing the conductor of $\psi$. 
For any finite prime $v$ of $k$, we put $N(v) = \# \F_v$, 
where $\F_v$ is the residue field of $v$. 
We then define the $T$-modified $L$-function by
$$L _T(s,\psi) = \parenth{ \prod_{v \in T} (1 - \psi(\varphi_v) N (v)^{1-s}) } L (s,\psi). $$
We define 
\begin{equation}\label{eq:omega}
\omega_T = \sum_{\psi} L_T (0, \psi) e_{\psi^{-1}} \in \Q[G],
\end{equation}
where $\psi$ runs over the characters of $G$ and 
$e_\psi = \frac{1}{\# G} \sum_{\sigma \in G} \psi(\sigma) \sigma^{-1}$
is the idempotent of the $\psi$-component.

Now the first main theorem of this paper is the following, whose proof will be given in \S \ref{sec:thm1}.

\begin{thm}\label{Main theorem 1}
Assume that eTNC for $K/k$ holds.
Then we have
\[ 
\Fitt_{\Z[G]^-} \parenth{ \Cl_K^{T, -} } =  
\parenth{\prod_{v \in S_{\ram} (K/k) \setminus S_{\infty}(k)} h_v^- \SFitt{1}_{\Z[G]^-} \parenth{ A_v^- } }
 \omega_T^-,  
\] 
where $\SFitt{1}_{\Z[G]^-}$ is the first shift of the Fitting ideal 
(see Definition \ref{defn:SFitt}). 
\end{thm}

In the second main result below, we will obtain a concrete description of $h_v^- \SFitt{1}_{\Z[G]^-} \parenth{ A_v^- }$, which completes the description of the Fitting ideal of $\Cl_K^{T, -}$.
We do not review the precise statement of eTNC (see e.g. \cite[Conjecture 3.6]{BKS16}).

In order to compare with Theorem \ref{Main theorem 1}, we recall the result for the dualized version:
\begin{equation}\label{eq:dual_Fitt}
 \Fitt_{\Z[G]^-} ( \Cl_K^{T, \vee, -} ) =  
\parenth{ \prod_{v \in S_{\ram} (K/k) \setminus S_{\infty}(k)}  
\parenth{\nu_{I_v} , 1-\cfrac{\nu_{I_v}}{\# I_v}\varphi_v^{-1} }^-} \omega_T^-.
\end{equation}
As already mentioned, Kurihara \cite[Corollary 3.7]{Kur20} showed this formula under the validity of the eTNC, and recently Dasgupta-Kakde \cite[Theorem 1.4]{DK20} removed the assumption.
Here, for a general $G$-module $M$, we equip the Pontryagin dual $M^{\vee}$ with the $G$-action by $(\sigma f)(x) = f(\sigma x)$ for $\sigma \in G$, $f \in M^{\vee}$, and $x \in M$.
This convention is the opposite of \cite{Kur20} and \cite{DK20}, so the right hand side of the formula \eqref{eq:dual_Fitt} differs from those by the involution.

We now briefly outline the proof of Theorem \ref{Main theorem 1}.
An important ingredient is an exact sequence of $\Z[G]^-$-modules of the form
$$
0 \longrightarrow  \mathfrak{A}^- \longrightarrow W_{S_\infty}^- \longrightarrow \Cl_K^{T, -}
\longrightarrow 0
$$
as in Proposition \ref{prop:fundamental}, where $\fA^-$ is free of finite rank $\# S'$.
Here, $S'$ is an auxiliary finite set of places of $k$.
This sequence was constructed by Kurihara \cite{Kur20}, based on preceding work such as Ritter-Weiss \cite{RW96} and Greither \cite{Gre07}, and played a key role in proving \eqref{eq:dual_Fitt} under eTNC.
Our novel idea is to construct an explicit injective homomorphism from $W_{S_{\infty}}^-$ to $(\Z[G]^-)^{\oplus \# S'}$ whose cokernel is isomorphic to the direct sum of $A_v^-$ for $v \in S' \setminus S_{\infty}(k)$.
Moreover, assuming eTNC, we will compute the determinant of the composite map $\fA^- \hookrightarrow W_{S_{\infty}}^- \hookrightarrow (\Z[G]^-)^{\oplus \# S'}$.
By using these observations, we obtain an exact sequence to which the theory of shifts of Fitting ideals can be applied, and then Theorem \ref{Main theorem 1} follows.

\subsection{Computation of the shift of Fitting ideal}\label{Intro2}

In order to make the formula of Theorem \ref{Main theorem 1} more explicit, in \S \ref{sec:Fitt1}, we will compute $\SFitt{1}_{\Z[G]} \parenth{A_v}$. 
This will be accomplished by using a similar method as Greither-Kurihara \cite[\S 1.2]{GK15}, which was actually a motivation for introducing the shifts of Fitting ideals in \cite{Kata_05}.

As the problem is purely algebraic, we deal with a general situation as follows (it should be clear from the notation how to apply the results below to the arithmetic situation; simply add subscripts $v$ appropriately).
Let $G$ be a finite abelian group.
Let $I$ and $D$ be subgroups of $G$ such that $I \subset D \subset G$ and that the quotient $D/I$ is a cyclic group.
We choose a generator $\varphi$ of $D/I$ and put
\[
g = 1 - \varphi^{-1} + \# I \in \Z[G/I],
\qquad h = 1 - \frac{\nu_{I}}{\# I}\varphi^{-1} + \nu_{I} \in \Q[G],
\]
which are non-zero-divisors.
We define a finite $\Z[G]$-module $A$ by
\[
A = \Z[G/I] / (g).
\]

In order to state the result, we introduce some notations.
We choose a decomposition
\begin{equation}\label{eq:I_decomp}
I = I_1 \times \cdots \times I_s
\end{equation}
as an abelian group such that $I_l$ is a cyclic group for each $1 \leq l \leq s$.
Here, we do not assume any minimality on this decomposition, so we allow even the extreme case where $I_l$ is trivial for some $l$.


For each $1 \leq l \leq s$, we put $\nu_l = \nu_{I_l} = \sum_{\sigma \in I_l} \sigma \in \Z[G]$.
We also put $\II_D = \Ker(\Z[G] \to \Z[G/D])$.

\begin{defn}\label{defn:J}
For $0 \leq i \leq s$, we define $Z_i$ as the ideal of $\Z[G]$ generated by $\nu_{l_1} \cdots \nu_{l_{s - i}}$ where $(l_1, \dots, l_{s - i})$ runs over all tuples of integers satisfying $1 \leq l_1 < \cdots < l_{s - i} \leq s$, that is,
\[
Z_i = (\nu_{l_1} \cdots \nu_{l_{s - i}} \mid 1 \leq l_1 < \cdots < l_{s - i} \leq s).
\]
We clearly have $Z_0 = (\nu_I) \subset Z_1 \subset \cdots \subset Z_s = (1)$.
We then define an ideal $\JJ$ of $\Z[G]$ by
\[
\JJ = \sum_{i = 1}^{s} Z_{i} \II_D^{i - 1}.
\]
\end{defn}

Note that the definition of $Z_i$ does depend on the choice of the decomposition \eqref{eq:I_decomp}.
On the other hand, it can be shown directly that the ideal $\JJ$ is independent from the choice.
We omit the direct proof because, at any rate, the independency can be deduced from the discussion in \S \ref{sec:Fitt1}.

\begin{eg}
When $s = 1$, we have $\JJ = (1)$.

When $s = 2$, we have 
\[
\JJ = (\nu_1, \nu_2) + \II_D.
\]

When $s = 3$, we have 
\[
\JJ = (\nu_1 \nu_2, \nu_2 \nu_3, \nu_3 \nu_1) + (\nu_1, \nu_2, \nu_3)\II_D + \II_D^2.
\]
\end{eg}

%

In this setting, we can describe $\SFitt{1}_{\Z[G]}(A)$ as follows. 
It is convenient to state the result after multiplying by $h$.

\begin{thm}\label{Main theorem 2}
We have
\[
h \SFitt{1}_{\Z[G]}(A) = \parenth{\nu_I, \parenth{1 - \frac{\nu_I}{\# I} \varphi^{-1} } \JJ}
\]
as fractional ideals of $\Z[G]$.
\end{thm}

\subsection{Stickelberger element and Fitting ideal}\label{Intro3}

As an application of Theorems \ref{Main theorem 1} and \ref{Main theorem 2}, 
we shall discuss the problem whether or not
the Stickelberger element lies in the Fitting ideal of $\Cl_K^{T, -}$.

We return to the setup in \S \ref{Intro1}.
Let $p$ be a fixed odd prime number and we shall work over $\Z_p$.
Let $G'$ denote the maximal subgroup of $G$ of order prime to $p$.
We put $k_p = K^{G'}$, which is the maximal $p$-extension of $k$ contained in $K$.
For each character $\chi$ of $G'$, we regard $\OO_{\chi} = \Z_p[\Imag(\chi)]$ as a $\Z_p[G']$-module via $\chi$, and put
$\Z_p[G]^{\chi} = \Z_p[G] \otimes_{\Z_p[G']} \OO_{\chi}$.
For a $\Z_p[G]$-module $M$, we put $M^{\chi} = M \otimes_{\Z_p[G]} \Z_p[G]^{\chi}$, which is an $\Z_p[G]^{\chi}$-module.
For an element $x \in M$, we write $x^{\chi}$ for the image of $x$ by the natural map $M \to M^{\chi}$.
We note that $\Z_p[G]$ is isomorphic to the direct product of $\Z_p[G]^{\chi}$ if $\chi$ runs over the equivalence classes of characters of $G'$.

From now on, we fix an odd character $\chi$ of $G'$.
We define $K_{\chi} = K^{\Ker(\chi)}$.
Then $K_{\chi}$ is a CM-field, $K_{\chi} \supset k_p$, and $K_{\chi}/k_p$ is a cyclic extension of order prime to $p$.

We put $S_{\chi} = S_{\ram}(K_{\chi}/k)$ and  
consider the $\chi$-component of the Stickelberger element defined by
\begin{equation}\label{eq:defn_theta}
\theta_{K/k, T}^{\chi} = \sum_{\psi|_{G'} = \chi} L_{S_{\chi}, T}(0, \psi) e_{\psi^{-1}} \in \Z_p[G]^{\chi},
\end{equation}
where $\psi$ runs over characters of $G$ whose restriction to $G'$ coincides with $\chi$ and we write
\[
L_{S_{\chi}, T}(s, \psi)
= \parenth{ \prod_{v \in S_{\chi} \setminus S_{\infty}(k)} (1 - \psi(\varphi_v))} 
	\parenth{ \prod_{v \in T} (1 - \psi(\varphi_v)N(v)^{1 - s})}
	L(s, \psi).
\]
Note that, comparing \eqref{eq:omega} and \eqref{eq:defn_theta}, we have
\begin{equation}\label{eq:theta}
\theta_{K/k, T}^{\chi}
= \parenth{\prod_{v \in S_{\chi} \setminus S_\infty(k)} \parenth{1 - \frac{\nu_{I_v}}{\# I_v} \varphi_v^{-1}}^{\chi}} \omega_T^{\chi}.
\end{equation}

Concerning the dualized version, by Dasgupta-Kakde \cite[Theorem 1.3]{DK20},
\[
\theta_{K/k, T}^{\chi} \in \Fitt_{\Z_p[G]^{\chi}} ( (\Cl_K^T \otimes \Z_p)^{\vee, \chi} )
\]
is always true.
This is called the strong Brumer-Stark conjecture.
More precisely, the displayed claim is a bit stronger than \cite[Theorem 1.3]{DK20} as we took $S_{\chi}$ instead of $S_{\ram}(K/k)$ in the definition of the Stickelberger element, but in any case it is an immediate consequence of the formula \eqref{eq:dual_Fitt}.

On the other hand, the corresponding claim without dual is known to be false in general (see \cite{GK08}).
However, we had only partial results and an exact condition was unknown.
The following theorem is strong as it gives a {\it necessary and sufficient} condition.

\begin{thm}\label{Main theorem 3}
Assume that eTNC for $K/k$ holds. 
Then, for each odd character $\chi$ of $G'$, the following are equivalent.
\begin{itemize}
\item[(i)]
We have $\theta_{K/k, T}^{\chi} \in \Fitt_{\Z_p[G]^{\chi}}((\Cl_K^T \otimes \Z_p)^{\chi})$.
\item[(ii)]
We have either $\theta_{K/k, T}^{\chi} = 0$ or, for any $v \in S_{\chi} \setminus S_{\infty}(k)$, one of the following holds.
\begin{itemize}
\item[(a)]
$v$ does not split completely in $K_{\chi}/k_p$.
\item[(b)]
The inertia group $I_v$ is cyclic.
\end{itemize}
\end{itemize}
\end{thm}

This theorem will be proved in \S \ref{sec:Stickel} as an application of Theorems \ref{Main theorem 1} and \ref{Main theorem 2}.
Note that there is an elementary equivalent condition for $\theta_{K/k, T}^{\chi} = 0$ as in Lemma \ref{lem:psi}.

Theorem \ref{Main theorem 3} indicates that the failure of the inertia groups to be cyclic is an obstruction for studying the Fitting ideal of the class group {\it without dual}.
The same phenomenon will appear again in Theorem \ref{thm:vs_dual} below.
We should say that this kind of phenomenon had been observed in preceding work, such as Greither-Kurihara \cite{GK08}.
It is also remarkable that the obstruction does not occur in the absolutely abelian case (i.e. when $k = \Q$), since in that case the inertia groups are automatically cyclic, apart from the $2$-parts.
This seems to fit the fact that Kurihara and Miura \cite{Kur03a}, \cite{KM11} succeeded in studying the class groups without dual in the absolutely abelian case.

Let us outline the proof of Theorem \ref{Main theorem 3}.
We assume that $\chi$ is a faithful character of $G'$ (i.e. $K_{\chi} = K$); actually we can deduce the general case from this case.
Since $\omega_T^{\chi}$ is a non-zero-divisor of $\Z_p[G]^{\chi}$, by Theorem \ref{Main theorem 1} and \eqref{eq:theta}, we have
$\theta_{K/k, T}^{\chi} \in \Fitt_{\Z_p[G]^{\chi}}((\Cl_K^T \otimes \Z_p)^{\chi})$ if and only if 
\begin{equation}\label{eq:e11}
\prod_{v} \parenth{1 - \frac{\nu_{I_v}}{\# I_v} \varphi_v^{-1}}^{\chi}
\subset \prod_{v} \parenth{ h_v \SFitt{1}_{\Z_p[G]}(A_v \otimes \Z_p)}^{\chi}
\end{equation}
holds as fractional ideals of $\Z_p[G]^{\chi}$, where on both sides $v$ runs over the elements of $S_{\ram}(K/k) \setminus S_{\infty}(k)$.

Obviously we may assume that $\theta_{K/k, T}^{\chi} \neq 0$.
The proof of (ii) $\Rightarrow$ (i) is the easier part.
We will show that, under the assumption (ii), the inclusion of \eqref{eq:e11} holds even for every $v$ before taking the product.
On the other hand, the opposite direction (i) $\Rightarrow$ (ii) is the harder part.
That is because, roughly speaking, we have to work over the ring $\Z_p[G]^{\chi}$, whose ring theoretic properties are not very nice.
A key idea to overcome this issue is to reduce to a computation in a discrete valuation ring.
More concretely, we make use of a character $\psi$ of $G$ which satisfies $\psi|_{G'} = \chi$ and a certain additional condition, whose existence is verified by Lemma \ref{lem:psi}, and we consider the $\Z_p[G]^{\chi}$-algebra $\OO_{\psi} = \Z_p[\Imag(\psi)]$.
By investigating the ideals in \eqref{eq:e11} after base change from $\Z_p[G]^{\chi}$ to $\OO_{\psi}$, we will show (i) $\Rightarrow$ (ii).


\subsection{Unconditional consequences}

Even if we do not assume the validity of eTNC, our argument shows the following.

\begin{thm}\label{thm:vs_dual}
We have an inclusion
\[
\Fitt_{\Z[G]^-} \parenth{ \Cl_K^{T, -} } \subset 
\Fitt_{\Z[G]^-} \parenth{ \Cl_K^{T, \vee, -} }. 
\]
Moreover, the inclusion is an equality if $I_v$ is cyclic for every $v \in S_{\ram}(K/k) \setminus S_{\infty}(k)$.
\end{thm}

This theorem follows immediately from Corollaries \ref{cor 1} and \ref{cor:Fitt1_vs_Fitt-1}.
Furthermore, by similar arguments as the proof of Theorem \ref{Main theorem 3}, we can observe that the inclusion is often proper.

As already remarked, Dasgupta-Kakde \cite{DK20} proved the formula \eqref{eq:dual_Fitt} {\it unconditionally}.
Therefore, if $I_v$ is cyclic for every $v \in S_{\ram} (K/k) \setminus S_{\infty}(k)$, we can also deduce from Theorem \ref{thm:vs_dual} that $\Fitt_{\Z[G]^-} \parenth{ \Cl_K^{T, -} }$ also coincides with that ideal, and this removes the assumption on eTNC in Theorem \ref{Main theorem 1}.
However, in Theorem \ref{Main theorem 1} we still need to assume eTNC when $I_v$ is not cyclic for some $v$.

\section{Definition of Fitting ideals and their shifts}\label{sec:SFitt}

In this section, we fix our notations concerning Fitting ideals.

\subsection{Fitting ideals}

Let $R$ be a noetherian ring.

\begin{defn}[cf. Northcott \cite{Nor76}]
We define the Fitting ideals as follows.
\begin{itemize}
\item[(i)]
Let $A$ be a matrix over $R$ with $m$ rows and $n$ columns.
For each integer $0 \leq i \leq n$, we define $\Fitt_{i, R}(A)$ as the ideal of $R$ generated by the $(n-i) \times (n-i)$ minors of $A$.
For each integer $i > n$, we also define $\Fitt_{i, R}(A) = (1)$.
\item[(ii)]
Let $X$ be a finitely generated $R$-module.
We choose a finite presentation $A$ of $X$ with $m$ rows and $n$ columns, that is, an exact sequence
\[
R^m \overset{\times A}{\to} R^n \to X \to 0.
\]
Here and henceforth, as a convention, we deal with row vectors, so we multiply matrices from the right.
Then, for each $i \geq 0$, we define the $i$-th Fitting ideal of $X$ by
\[
\Fitt_{i, R}(X) = \Fitt_{i, R}(A).
\]
It is known that this ideal does not depend on the choice of $A$.
When $i = 0$, we also write $\Fitt_{R}(X) = \Fitt_{0, R}(X)$ and call it the initial Fitting ideal.
\end{itemize}

\end{defn}

We will later make use of the following elementary lemma.
We omit the proof (cf. Kurihara \cite[Lemma 3.3]{Kur03b}).

\begin{lem}\label{lem:hFitt}
Let $X$ be a finitely generated $R$-module and $I$ be an ideal of $R$.
If $X$ is generated by $n$ elements over $R$, then
\[
\Fitt_{0, R}(X/IX) = \sum_{i = 0}^{n} I^i \Fitt_{i, R}(X).
\]
\end{lem}


\subsection{Shifts of Fitting ideals}\label{subsec:SFitt}

In this subsection, we review the definition of shifts of Fitting ideals introduced by the second author \cite{Kata_05}.

Although we can deal with a more general situation, for simplicity we consider the following.
Let $\Lambda$ be a Dedekind domain (e.g. $\Lambda = \Z$, $\Z[1/2]$, or $\Z_p$).
Let $\Delta$ be a finite abelian group and consider the ring $R = \Lambda[\Delta]$.

We define $\CC$ as the category of $R$-modules of finite length.
We also define a subcategory $\PP$ of $\CC$ by
\[
\PP = \{P \in \CC \mid \pd_{R}(P) \leq 1\},
\]
where $\pd_{R}$ denotes the projective dimension over $R$.
Note that any module $M$ in $\CC$ satisfies $\pd_{\Lambda}(M) \leq 1$.

\begin{defn}\label{defn:SFitt}
Let $X$ be an $R$-module in $\CC$ and $d \geq 0$ an integer.
We take an exact sequence
\[
0 \to Y \to P_1 \to \dots \to P_d \to X \to 0
\]
in $\CC$ with $P_1, \dots, P_d \in \PP$.
Then we define
\[
\SFitt{d}_R(X) = \parenth{\prod_{i=1}^d \Fitt_R(P_i)^{(-1)^i}} \Fitt_R(Y).
\]
The well-definedness (i.e. the independence from the choice of the $n$-step resolution) is proved in \cite[Theorem 2.6 and Proposition 2.7]{Kata_05}.
\end{defn}

%

We also introduce a variant for the case where $d$ is negative.

\begin{defn}\label{defn:SFittN}
Let $X$ be an $R$-module in $\CC$ and $d \leq 0$ an integer.
We take an exact sequence
\[
0 \to X \to P_1 \to \dots \to P_{-d} \to Y \to 0
\]
in $\CC$ with $P_1, \dots, P_{-d} \in \PP$.
Then we define
\[
\SFittN{d}_R(X) = \parenth{\prod_{i=1}^{-d} \Fitt_R(P_i)^{(-1)^i}} \Fitt_R(Y).
\]
The well-definedness is proved in \cite[Theorem 3.19 and Propositions 2.7 and 3.17]{Kata_05}.
\end{defn}

\section{Fitting ideals of ideal class groups}\label{sec:thm1}

In this section, we prove Theorem \ref{Main theorem 1}, which describes the Fitting ideal of $\Cl_K^{T, -}$
using shifts of Fitting ideals. 
We keep the notation in \S \ref{Intro1}.

\subsection{Brief review of work of Kurihara}

We first review necessary ingredients from Kurihara \cite{Kur20}, which in turn relies on preceding work, in particular Ritter-Weiss \cite{RW96} and Greither \cite{Gre07}.

For each place $w$ of $K$, 
let $D_w$ and $I_w$ denote the decomposition subgroup and the inertia subgroup of $w$ in $G$, respectively. 
These subgroups depend only on the place of $k$ which lies below $w$.

Let us introduce local modules $W_v$.
For any finite group $H$, we define $\Delta H$ as the augmentation ideal in $\Z [H]$. 

\begin{defn}\label{defn:W_v}
For each finite prime $w$ of $K$, we define a $\Z[D_w]$-module $W_{K_w}$ by 
\begin{equation}\label{eq:defn_W}
W_{K_w} = \{(x, y) \in \Delta D_w \oplus \Z[D_w / I_w] \mid \ol{x} = (1 - \varphi_v^{-1}) y \},
\end{equation}
where $\bar{x}$ denotes the image of $x$ in $\Z[D_w/I_w]$.
For each finite prime $v$ of $k$, 
we define the $\Z[G]$-module $W_v$ by taking the direct sum as 
\[
W_v = \bigoplus_{w | v} W_{K_w},
\] 
where $w$ runs over the finite primes of $K$ which lie above $v$.
Alternatively, $W_v$ can be regarded as the induced module of $W_{K_w}$ from $D_w$ to $G$, as long as we choose a place $w$ of $K$ above $v$.
\end{defn}

We take an auxiliary finite set $S'$ of places of $k$ satisfying the following conditions.
\begin{itemize}
\item
$S' \supset S_{\ram} (K/k)$.
\item
$S' \cap T = \emptyset$.
\item
$\Cl_{K, S'}^T = 0$, 
where $\Cl_{K, S'}^T = \Coker \parenth{K_T^\times \overset{\oplus \ord_w}{\longrightarrow} 
\bigoplus_{w \notin S'_K \cup T_K} \Z}$.
\item
$G$ is generated by the decomposition groups $D_v$ of $v$ for all $v \in S'$.
\end{itemize}
We define a $\Z[G]$-module $W_{S_\infty}$ by
$$ W_{S_\infty} = \bigoplus_{w \in S_{\infty}(K)} \Delta D_w \oplus \bigoplus_{v \in S' \setminus S_\infty(k)} W_v.$$
By using local and global class field theory, Kurihara constructed an exact sequence of the following form.

\begin{prop}[{Kurihara \cite[\S 2.2, sequence (5)]{Kur20}}]\label{prop:fundamental}
We have an exact sequence
$$
0 \longrightarrow  \mathfrak{A}^- \longrightarrow W_{S_\infty}^- \longrightarrow \Cl_K^{T, -}
\longrightarrow 0 , 
$$
where $\mathfrak{A}^-$ is a free $\Z [G]^-$-module of rank $\# S'$.
\end{prop}

In \cite{Kur20}, the author took the linear dual of this sequence, and the resulting sequence played an important role to study $\Cl_K^{T, \vee, -}$.
In this paper, we do not take the linear dual but instead study the sequence itself for the proof of Theorem \ref{Main theorem 1}.

\subsection{Definition of $f_v$}

Our key ingredient for the proof of Theorem \ref{Main theorem 1} is the following homomorphism $f_v$.

\begin{defn}
For a finite prime $w$ of $K$, we define a $\Z[D_w]$-homomorphism $$f_w : W_{K_w} \longrightarrow \Z[D_w]$$ by 
$f_w(x,y) = x + \nu_{I_w} (y)$ (recall the definition of $W_{K_w}$ in Definition \ref{defn:W_v}).
For a finite prime $v$ of $k$, 
we then define a $\Z[G]$-homomorphism 
$f_v : W_v \longrightarrow \Z [G]$ by 
\begin{eqnarray}\label{f_v}
f_v : W_v = \bigoplus_{w | v} W_{K_w} \overset{ \oplus f_w}{\longrightarrow} \bigoplus_{w|v} \Z[D_w] \simeq \Z[G],\end{eqnarray} 
where the last isomorphism depends on a choice of $w$.
\end{defn}

In \S \ref{Intro1} we introduced a finite $\Z[G]$-module $A_v = \Z[G/I_v]/(g_v)$ with $g_v = 1 - \varphi_v^{-1} + \# I_v$.
It is actually motivated by the following.

\begin{lem} \label{lemma A_v}
For any finite prime $v$ of $k$, the map $f_v$ is injective and 
$$\Coker f_v \simeq A_v.$$
\end{lem}

\begin{proof}
It is enough to show that $f_w$ is injective and $\Coker f_w \simeq \Z [D_w / I_w] / (g_v)$ 
for any finite prime $w$ of $K$.
Put $J_w = \Ker (\Z[D_w] \longrightarrow \Z[D_w / I_w])$. 
We define a homomorphism $\alpha_w: J_w \to W_{K_w}$ by $\alpha_w(x) = (x,0)$.
Let us consider the following commutative diagram 
$$
\begin{CD}
0 @>>> J_w @>\alpha_w>> W_{K_w} @>>> \Coker \; \alpha_w @>>> 0 @. \\
@. @| @VVf_wV @VVf'_wV @.  \\
 0  @>>> J_w @>>> \Z[D_w] @>>> \Z[D_w/I_w] @>>> 0 @.  , \\
\end{CD}
$$
where the lower sequence is the trivial one, the commutativity of the left square is easy, and the right vertical arrow is the induced one.
By the definition of $W_{K_w}$, we have
\begin{eqnarray*}
\Coker \; \alpha_w = 
\{ (\bar{x},y) \in \Delta (D_w / I_w) \times \Z[D_w / I_w] \mid \ol{x} = (1 - \varphi_v^{-1})y  \}. 
\end{eqnarray*}
Since $D_w / I_w$ is a cyclic group generated by $\varphi_v^{-1}$, 
the $\Z[D_w / I_w]$-module $\Coker \; \alpha_w $ is free of rank $1$ with a basis $(1- \varphi_v^{-1}, 1)$.
Moreover, $f_w'$ sends this basis to $g_v = 1 - \varphi_v^{-1} + \# I_v$.
Therefore, $f_w'$ is injective with cokernel isomorphic to $\Z[D_w/I_w]/(g_v)$.
Then by the diagram $f_w$ also satisfies the desired properties.
\end{proof}

For any $v \in S' \setminus S_\infty(k)$, 
we consider the homomorphism $f_v^- : W_v^- \longrightarrow \Z[G]^-$ which is induced by $f_v$.  
For any $v \in S_\infty(k)$, 
we have $( \oplus_{w \mid v} \Delta D_w)^- \simeq \Z[G]^-$ by choosing $w$, so we fix this isomorphism and write $f_v^-$ for it. 
Using these $f_v^-$, 
we consider the following commutative diagram
$$
\begin{CD}
0 @>>>  \mathfrak{A}^- @>>> W_{S_\infty}^- @>>> \Cl_K^{T, -} @>>> 0 @. \\
@. @|  @VV\oplus_{v \in S'} f_v^-V @VVV @.  \\
0 @>>> \mathfrak{A}^- @>\psi>> \bigoplus_{v \in S'} \Z[G]^- @>>> \Coker \psi  @>>> 0 @.  , \\
\end{CD}
$$
where the upper sequence is that in Proposition \ref{prop:fundamental} and the map $\psi$ is defined by the commutativity.
By Lemma \ref{lemma A_v} and the snake lemma, we get the following proposition.  

\begin{prop}\label{coker psi} 
We have an exact sequence
$$ 0 \longrightarrow  \Cl_K^{T, -} 
\longrightarrow  \Coker  \psi 
\longrightarrow \bigoplus_{v \in S' \setminus S_\infty(k)}A_v^- 
\longrightarrow 0. $$
Moreover, the $\Z[G]^-$-module $\Coker \psi$ is finite with $\pd_{\Z[G]^-}(\Coker \psi) \leq 1$.
\end{prop}

By Proposition \ref{coker psi}, the Fitting ideal $\Fitt_{\Z[G]^-} (\Coker  \psi)$ is a principal ideal of $\Z[G]^-$ generated by a non-zero-divisor.
Then we can describe the Fitting ideals of $\Cl_K^{T,-}$ and of $\Cl_K^{T, \vee, -}$ as follows. 

\begin{cor} \label{cor 1}
We have 
\[ \Fitt_{\Z[G]^-} ( \Cl_K^{T, -} ) 
= \Fitt_{\Z[G]^-} (\Coker \psi) 
\prod_{v \in S' \setminus S_{\infty}(k)} \SFitt{1}_{\Z[G]^-} \parenth{A_v^-}
\]
and
\[ \Fitt_{\Z[G]^-} ( \Cl_K^{T, \vee, -} ) 
= \Fitt_{\Z[G]^-} (\Coker \psi) 
\prod_{v \in S' \setminus S_{\infty}(k)} \SFittN{-1}_{\Z[G]^-} \parenth{A_v^-}.
\]
\end{cor}

\begin{proof}
The first formula follows directly from Proposition \ref{coker psi} and Definition \ref{defn:SFitt}.
For the second formula, by \cite[Proposition 4.7]{Kata_05}, we have
\[
\Fitt_{\Z[G]^-}(\Cl_K^{T, \vee, -}) 
= \SFittN{-2}_{\Z[G]^-}(\Cl_K^{T, -}).
\]
By Proposition \ref{coker psi} and Definition \ref{defn:SFittN}, we also have
\begin{align}
\SFittN{-2}_{\Z[G]^-}(\Cl_K^{T, -})
& = \Fitt_{\Z[G]^-} (\Coker \psi) \prod_{v \in S' \setminus S_\infty(k)} \SFittN{-1}_{\Z[G]^-} \parenth{A_v^-}.
\end{align}
This completes the proof.
\end{proof}

\subsection{Fitting ideal of $\Coker \psi$}\label{subsec:Fitt psi}

Recall the definitions of $\omega_T$ and of $h_v$ in \S \ref{Intro1}.

\begin{thm} \label{Kurihara}
Assume that eTNC for $K/k$ holds. 
Then we have
\[
\Fitt_{\Z[G]^-} (\Coker  \psi) = \parenth{ \parenth{\prod_{v \in S' \setminus S_\infty(k)} h_v^-}  \omega_T^- }.
\]
\end{thm}

\begin{proof}
For each $v \in S' \setminus S_{\infty}(k)$,
we define a basis $e_v$ of $\Hom_{\Q[G]}(W_v \otimes \Q, \Q[G])$ as in \cite[\S 2.2, equation (9)]{Kur20} (we do not recall the precise definition here).
Then we can see that its dual basis $e_v'$ of $W_v \otimes \Q$ is given by
\[
e_v' = \frac{1}{1 - \wtil{\varphi_v}^{-1} + N_{I_v}} (1 - \wtil{\varphi_v}^{-1}, 1),
\]
where $\wtil{\varphi_v}$ is a lift of $\varphi_v$.
Then, by the definition of $f_v$, this element satisfies $f_v(e_v') = 1$, where by abuse of notation $f_v$ denotes the homomorphism $W_v \otimes \Q \to \Q[G]$ induced by $f_v: W_v \to \Z[G]$.
For $v \in S_{\infty}(k)$, as a basis over $\Z[G]^-$, we take the element $e_v^{\prime, -}$ of $\parenth{\bigoplus_{w \mid v} \Delta D_w}^-$ which is characterized by $f_v^-(e_v^{\prime, -}) = 1$.

Let us consider the isomorphism $\Psi: \fA^- \otimes \Q \to W_{S_{\infty}}^- \otimes \Q$ induced by the sequence in Proposition \ref{prop:fundamental}.
Then, under eTNC, Kurihara \cite[Theorem 3.6]{Kur20} proved
\[
\det(\Psi) = \parenth{\prod_{v \in S' \setminus S_{\infty}(k)} h_v^- } \omega_T^-
\]
with respect to a certain basis of $\fA^-$ as a $\Z[G]^-$-module and the basis $(e_v^{\prime, -})_{v \in S'}$ of $W_{S_{\infty}}^-$.
Actually this is an easy reformulation of the result of Kurihara, which concerns the determinant of the linear dual of $\Psi$.

Therefore, the determinant of the composite map $\psi$ of $\Psi$ and $\bigoplus_{v \in S'} f_v^-$, with respect to the basis of $\fA^-$ and the standard basis of $(\Z[G]^-)^{\oplus \# S'}$, also coincides with $\parenth{\prod_{v \in S' \setminus S_{\infty}(k)} h_v^- } \omega_T^-$.
This shows the theorem.
\end{proof}

We are now ready to prove Theorem \ref{Main theorem 1}.
%
%
%

\begin{proof}[Proof of Theorem \ref{Main theorem 1}]
By Corollary \ref{cor 1} and Theorem \ref{Kurihara}, 
we have 
\[ \Fitt_{\Z[G]^-} ( \Cl_K^{T, -} ) = \parenth{
\prod_{v \in S' \setminus S_\infty(k)} h_v^- \SFitt{1}_{\Z[G]^-}  \parenth{ A_v^-} } \omega_T^-. 
\]
For $v \in S' \setminus S_{\ram}(K/k)$, 
we have $A_v = \Z [G] / (h_v)$, so
$$ \SFitt{1}_{\Z[G]^-} \big( A_v^- \big) = (h_v^-)^{-1}. $$
Then Theorem \ref{Main theorem 1} follows.
\end{proof}

\begin{rem}\label{rem:dual_Fitt}
Similarly, under the validity of eTNC, Corollary \ref{cor 1} and Theorem \ref{Kurihara} also imply a formula
\[ \Fitt_{\Z[G]^-} ( \Cl_K^{T, \vee, -} ) 
= \parenth{
\prod_{v \in S_{\ram} (K/k) \setminus S_{\infty}(k)} h_v^- \SFittN{-1}_{\Z[G]^-}  \parenth{ A_v^-}} \omega_T^-. 
\]
Combining this with Proposition \ref{prop:Fitt-1} below, we can recover the formula \eqref{eq:dual_Fitt}.
This argument may be regarded as a reinterpretation of the work \cite{Kur20} by using the shifts of Fitting ideals.
\end{rem}

\section{Computation of shifts of Fitting ideals}\label{sec:Fitt1}

In this section, we prove Theorem \ref{Main theorem 2} on the description of $\SFitt{1}_{\Z[G]}(A)$.
We keep the notations as in \S \ref{Intro2}.

\subsection{Computation of $\SFittN{-1}_{\Z[G]}(A)$}\label{subsec:Fitt-1}

Before $\SFitt{1}_{\Z[G]}(A)$, we determine $\SFittN{-1}_{\Z[G]}(A)$, which is actually much easier.

We choose a lift $\wtil{\varphi} \in D$ of $\varphi$ and put
\[
\wtil{g} = 1 - \wtil{\varphi}^{-1} + \# I \in \Z[G],
\]
which is again a non-zero-divisor.
Obviously, $g$ is then the natural image of $\wtil{g}$ to $\Z[G/I]$.

\begin{prop}\label{prop:Fitt-1}
We have
\[
\SFittN{-1}_{\Z[G]}(A) = \parenth{1, \nu_I g^{-1}}.
\]
Therefore, we also have
\[
h \SFittN{-1}_{\Z[G]}(A) 
= \parenth{\nu_I, 1 - \frac{\nu_I}{\# I} \varphi^{-1} }.
\]
\end{prop}

\begin{proof}
We have an exact sequence
\[
0 \to \Z[G/I] \overset{\nu_I}{\to} \Z[G] \to \Z[G]/(\nu_I) \to 0.
\]
Since multiplication by $\wtil{g}$ is injective on each of these modules, applying the snake lemma, we obtain an exact sequence
\[
0 \to A \to \Z[G]/(\wtil{g}) \to \Z[G]/(\wtil{g}, \nu_I) \to 0.
\]
By Definition \ref{defn:SFittN}, we then have
\[
\SFittN{-1}_{\Z[G]}(A) 
= (\wtil{g})^{-1} (\wtil{g}, \nu_I)
= \parenth{1, \nu_I g^{-1}}.
\]
This proves the former formula of the proposition.

Since we have $\nu_I g = \nu_I h$, the former formula implies 
$h \SFittN{-1}_{\Z[G]}(A) = \parenth{\nu_{I}, h}$.
Then the latter formula follows from $h \equiv 1 - \frac{\nu_{I}}{\# I} \varphi^{-1} (\bmod (\nu_I))$.
\end{proof}

Before proving Theorem \ref{Main theorem 2}, we show a corollary.

\begin{cor}\label{cor:Fitt1_vs_Fitt-1}
We have an inclusion
\[
\SFitt{1}_{\Z[G]}(A) \subset \SFittN{-1}_{\Z[G]}(A).
\]
Moreover, if $I$ is a cyclic group, the inclusion is an equality.
\end{cor}

\begin{proof}
By Definition \ref{defn:J}, the ideal $\JJ$ is contained in $\Z[G]$ and we have $\JJ = \Z[G]$ if $I$ is cyclic.
Hence this corollary immediately follows from Theorem \ref{Main theorem 2} and Proposition \ref{prop:Fitt-1}.
\end{proof}

\subsection{Computation of $\SFitt{1}_{\Z[G]}(A)$}\label{subsec:Fitt1}

This subsection is devoted to the proof of Theorem \ref{Main theorem 2}.

We fix the decomposition \eqref{eq:I_decomp} of $I$.
For each $1 \leq l \leq s$, we choose a generator $\sigma_l$ of $I_l$ and put $\tau_l = \sigma_l - 1 \in \Z[G]$.
Note that we then have $\nu_l = 1 + \sigma_l + \sigma_l^2 + \dots + \sigma_l^{\# I_l - 1}$ and $\tau_l \nu_l = 0$.
As in \S \ref{subsec:Fitt-1}, we put $\wtil{g} = 1 - \wtil{\varphi}^{-1} + \# I$ after choosing $\wtil{\varphi}$.

We recall $\II_D = \Ker(\Z[G] \to \Z[G/D])$ and also put $\II_I = \Ker(\Z[G] \to \Z[G/I])$.
Then we have $\II_I = (\tau_1, \dots, \tau_s)$ and $\II_D = (\II_I, 1 - \wtil{\varphi}^{-1})$.

We begin with a proposition.

\begin{prop}\label{prop:01}
We have
\[
\SFitt{1}_{\Z[G]}(A) = \sum_{i = 0}^{s} \wtil{g}^{i-1} \Fitt_{i, \Z[G]}(\II_I).
\]
\end{prop}

\begin{proof}
We have the tautological exact sequence
\[
0 \to \II_I \to \Z[G] \to \Z[G/I] \to 0.
\]
Since multiplication by $\wtil{g}$ is injective on each of these modules, by applying snake lemma, we obtain an exact sequence
\[
0 \to \II_I/\wtil{g} \II_I \to \Z[G]/(\wtil{g}) \to A \to 0.
\]
Then Definition \ref{defn:SFitt} implies
\begin{align}
\SFitt{1}_{\Z[G]}(A)
& = \wtil{g}^{-1} \Fitt_{\Z[G]}(\II_I/\wtil{g} \II_I).
\end{align}
Since $\II_I$ is generated by the $s$ elements $\tau_1, \dots, \tau_s$, we have
\[
\Fitt_{\Z[G]}(\II_I/\wtil{g} \II_I) 
= \sum_{i = 0}^{s} \wtil{g}^i \Fitt_{i, \Z[G]}(\II_I)
\]
by Lemma \ref{lem:hFitt}.
Thus we obtain the proposition.
\end{proof}

Our next task is to determine $\Fitt_{i, \Z[G]}(\II_I)$ for $0 \leq i \leq s$.
The result will be Proposition \ref{prop:FittB} below.
For that purpose, we construct a concrete free resolution of $\Z$ over $\Z[I]$, using an idea of Greither-Kurihara \cite[\S 1.2]{GK15} (one may also refer to \cite[\S 4.3]{Kata_05}).

For each $1 \leq l \leq s$, we have a homological complex
\[
C^l: \cdots \overset{\tau_l}{\to} \Z[I_l] \overset{\nu_l}{\to} \Z[I_l] \overset{\tau_l}{\to} \Z[I_l] \to 0
\]
over $\Z[I_l]$, concentrated at degrees $\geq 0$.
Let $C^l_n$ be the degree $n$ component of $C^l$, so $C^l_n = \Z[I_l]$ if $n \geq 0$ and $C^l_n = 0$ otherwise.
Then the homology groups are $H_n(C^l) = 0$ for $n \neq 0$ and $H_0(C^l) \simeq \Z$.

We define a complex $C$ over $\Z[I]$ by
\[
C = \bigotimes_{l=1}^s C^l,
\]
which is the tensor product of complexes over $\Z$ 
(we do not specify the sign convention as it does not matter to us; we define it appropriately so that the descriptions of $d_1$ and $d_2$ below are valid).
Explicitly, the degree $n$ component $C_n$ of $C$ is defined as
\[
C_n = \bigoplus_{n_1 + \dots + n_s = n} C^1_{n_1} \otimes \dots \otimes C^s_{n_s}.
\]
Clearly the tensor product is zero unless $n_1, \dots, n_s \geq 0$, and in that case 
\[
C^1_{n_1} \otimes \dots \otimes C^s_{n_s} = \Z[I_1] \otimes \dots \otimes \Z[I_s] \simeq \Z[I].
\]
It is convenient to write $x_1^{n_1} \cdots x_s^{n_s}$ for the basis of $C^1_{n_1} \otimes \dots \otimes C^s_{n_s}$ for each $n_1, \dots, n_s \geq 0$, following \cite{GK15}.
Then, for each $n \geq 0$, the module $C_n$ is a free module on the set of monomials of $x_1, \dots, x_s$ of degree $n$.

A basic property of tensor products of complexes implies that $H_n(C) = 0$ for $n \neq 0$ and $H_0(C) \simeq \Z$.
Therefore, $C$ is a free resolution of $\Z$ over $\Z[I]$.

It will be necessary to investigate some components of $C$ of low degrees.
Note that $C_0$ is free of rank one with a basis $1 (= x_1^0 \cdots x_s^0)$, $C_1$ is a free module on the set
\[
S_1 = \{x_1, \dots, x_s\},
\]
and $C_2$ is a free module on the set $S_2 \cup S_2'$ where
\[
S_2 = \{x_1^2, \dots, x_s^2\},
\qquad S_2' = \{x_l x_{l'} \mid 1 \leq l < l' \leq s\}.
\]
Moreover, the differential $d_n: C_n \to C_{n-1}$ for $n=1, 2$ are described as follows.
We have 
\[
d_1(x_l) = \tau_l \cdot 1
\]
for each $1 \leq l \leq s$,
\[
d_2(x_l^2) = \nu_l x_l
\]
for each $1 \leq l \leq s$, and 
\[
d_2(x_l x_{l'}) = \tau_l x_{l'} - \tau_{l'} x_l
\]
for each $1 \leq l < l' \leq s$.

Let $M$ denote the presentation matrix of $d_2$.
For clarity, we define $M$ formally as follows.

\begin{defn}\label{defn:MandN}
We define a matrix 
\[
M = M_s(\nu_1,\dots, \nu_s, \tau_1, \dots, \tau_s)
\]
with the columns (resp. the rows) indexed by $S_1$ (resp. $S_2 \cup S_2'$), by
\[
\begin{cases}
\text{the $(x_l^2, x_l)$ component is $\nu_l$} & \text{for $1 \leq l \leq s$,}\\
\text{the $(x_l x_{l'}, x_l)$ component is $-\tau_{l'}$} & \text{for $1 \leq l < l' \leq s$,}\\
\text{the $(x_l x_{l'}, x_{l'})$ component is $\tau_l$} & \text{for $1 \leq l < l' \leq s$,}\\
\text{and the other components are zero.}
\end{cases}
\]
Here, we do not specify the orders of the sets $S_1$ and $S_2 \cup S_2'$.
The ambiguity does not matter for our purpose.

For later use, we also define a matrix
\[
N_s(\tau_1, \dots, \tau_s)
\]
as the submatrix of $M$ with the rows in $S_2$ removed.
More precisely, we define the matrix $N_s(\tau_1, \dots, \tau_s)$ with the columns (resp. rows) indexed by $S_1$ (resp. $S_2'$), by
\[
\begin{cases}
\text{the $(x_l x_{l'}, x_l)$ component is $-\tau_{l'}$} & \text{for $1 \leq l < l' \leq s$,}\\
\text{the $(x_l x_{l'}, x_{l'})$ component is $\tau_l$} & \text{for $1 \leq l < l' \leq s$,}\\
\text{and the other components are zero.}
\end{cases}
\]
Therefore, by choosing appropriate orders of rows and columns, we have
\[
M_s(\nu_1,\dots, \nu_s, \tau_1, \dots, \tau_s) = 
\begin{pmatrix}
\begin{matrix}
\nu_1 & & \\
& \ddots & \\
& & \nu_s \\
\end{matrix}\\
 N_s(\tau_1, \dots, \tau_s) 
\end{pmatrix}.
\]
\end{defn}

\begin{eg}
When $s = 3$, we have
\[
M = 
\begin{pmatrix}
\nu_1 & & \\
& \nu_2 & \\
& & \nu_3 \\
& - \tau_3 & \tau_2\\
- \tau_3 & & \tau_1\\
- \tau_2 & \tau_1 &
\end{pmatrix}
\]
Here, we use the order $x_2x_3, x_1x_3, x_1x_2$ for the set $S_2'$.
\end{eg}

\begin{prop}\label{prop:04}
The matrix $M_s(\nu_1,\dots, \nu_s, \tau_1, \dots, \tau_s)$, over $\Z[G]$, is a presentation matrix of the module $\II_I$.
\end{prop}

\begin{proof}
By the construction, $M$ is a presentation matrix of $\Ker(\Z[I] \to \Z)$ over $\Z[I]$.
Since $\Z[G]$ is flat over $\Z[I]$, we obtain the proposition by base change.
\end{proof}

\begin{prop}\label{prop:02}
For each $0 \leq i \leq s$, we have
\[
\Fitt_{i, \Z[G]}(M) 
= \sum_{j=0}^{s - i} \sum_{\substack{a \subset \{1, 2, \dots, s\} \\ \# a = j}} \nu_{a_1} \cdots \nu_{a_j} \Fitt_{i, \Z[G]} (N_{s - j}(\tau_{a_{j + 1}}, \dots, \tau_{a_s})).
\]
Here, for each $j$, in the second summation $a$ runs over subsets of $\{1, 2, \dots, s\}$ of $j$ elements, and for each $a$ we define $a_1, \dots, a_s$ by requiring 
\[
a = \{a_1, \dots, a_j\},
\qquad \{a_1, \dots, a_s\} = \{1, 2, \dots, s\}, 
\qquad a_1 < \dots < a_j, 
\qquad a_{j + 1} < \dots < a_s.
\]
The matrix $N_{s - j}(\tau_{a_{j + 1}}, \dots, \tau_{a_s})$ is defined as in Definition \ref{defn:MandN} for $s - j$ and $\tau_{a_{j+1}}, \dots, \tau_{a_s}$ instead of $s$ and $\tau_1, \dots, \tau_s$.
\end{prop}

\begin{proof}
By the definition of higher Fitting ideals, $\Fitt_{i, \Z[G]}(M)$ is generated by $\det(H)$ for square submatrices $H$ of $M$ of size $s - i$.
Such a matrix $H$ is in one-to-one correspondence with choices of a subset $A_H^{\column} \subset S_1 = \{x_1, \dots, x_s \}$ with $\# A_H^{\column} = s - i$ and a subset $A_H^{\row} \subset S_2 \cup S_2' = \{x_1^2, \dots, x_s^2, x_1x_2, \dots, x_{s-1}x_s\}$ with $\# A_H^{\row} = s - i$.
We only have to deal with $H$ satisfying $\det(H) \neq 0$.

For each $H$, we define $j$ and $a$ by
\[
j = \#(A_H^{\row} \cap S_2)
\]
(so clearly $0 \leq j \leq s - i$) and
\[
A_H^{\row} \cap S_2 = \{x_{a_1}^2, \dots, x_{a_j}^2\}.
\]
Recall that the $x_l^2$ row in the matrix $M$ contains a unique non-zero component $\nu_l$ in the $x_l$ column.
Therefore, the assumption $\det(H) \neq 0$ forces $x_{a_1}, \dots, x_{a_j} \in A_H^{\column}$ and 
\[
\det(H) = \pm \nu_{a_1} \cdots \nu_{a_j} \det(H'),
\]
where $H'$ is the square submatrix of $H$ of size $(s - i) - j$, with rows in $A_{H'}^{\row} = A_H^{\row} \setminus \{x_{a_1}^2, \dots, x_{a_j}^2\} = A_H^{\row} \cap S_2'$ and columns in $A_{H'}^{\column} = A_H^{\column} \setminus \{x_{a_1}, \dots, x_{a_j}\}$.

Let $N_a$ denote the submatrix of $N_s(\tau_1, \dots, \tau_s)$ obtained by removing the $x_{a_1}, \dots, x_{a_j}$ columns.
Then it is clear that $\det(H')$'s (for fixed $j$ and $a$) as above generate $\Fitt_{i, \Z[G]} (N_a)$.
The argument so far implies
\[
\Fitt_{i, \Z[G]}(M) 
= \sum_{j=0}^{s - i} \sum_{\substack{a \subset \{1, 2, \dots, s\} \\ \# a = j}} \nu_{a_1} \cdots \nu_{a_j} \Fitt_{i, \Z[G]} (N_a).
\]
By the formula $\tau_l \nu_l = 0$, we may remove the components $\pm \tau_{a_1}, \dots, \pm \tau_{a_j}$ from the matrix $N_a$ in the right hand side.
It is easy to check that the resulting matrix is nothing but $N_{s - j}(\tau_{a_{j + 1}}, \dots, \tau_{a_s})$ (with several zero rows added).
This completes the proof.
\end{proof}

\begin{prop}\label{prop:03}
For $s \geq 0$ and $i \geq 0$, we have 
\[
\Fitt_{i, \Z[G]}(N_s(\tau_1, \dots, \tau_s)) = 
\begin{cases}
(1) & (i \geq s)\\
0 & (s \geq 1, i = 0)\\
(\tau_1, \dots, \tau_s)^{s - i} & (1 \leq i < s)\\
\end{cases}
\]
\end{prop}

\begin{proof}
Since $N_s(\tau_1, \dots, \tau_s)$ has $s$ columns, the case for $i \geq s$ is clear.

We show the vanishing when $s \geq 1$ and $i = 0$.
Let $R = \Z[T_1, \dots, T_s]$ be the polynomial ring over $\Z$.
Then we have a ring homomorphism $f: R \to \Z[G]$ defined by sending $T_l$ to $\tau_l$.
We define a matrix $N_s(T_1, \dots, T_s)$ over $R$ in the same way as in Definition \ref{defn:MandN}, with $\tau_{\bullet}$ replaced by $T_{\bullet}$.
Then, by the base change via $f$, we have
\[
\Fitt_{\Z[G]}(N_s(\tau_1, \dots, \tau_s)) = f(\Fitt_{R}(N_s(T_1, \dots, T_s))) \Z[G].
\]
Hence the left hand side would vanish if we show that $\Fitt_{R}(N_s(T_1, \dots, T_s)) = 0$.

For each $1 \leq l \leq s$, we consider the complex
\[
\wtil{C}^l: 0 \to \Z[T_l] \overset{T_l}{\to} \Z[T_l] \to 0,
\]
over $\Z[T_l]$, which satisfies $H_n(\wtil{C}^l) = 0$ for $n \neq 0$ and  $H_0(\wtil{C}^l) \simeq \Z$.
Similarly as previous, by taking the tensor product of the complexes $\wtil{C}^l$ over $\Z$, we obtain an exact sequence
\[
\cdots \to \wtil{C}_2
 \overset{N_s(T_1, \dots, T_s)}{\to} \wtil{C}_1
\overset{\tiny \begin{pmatrix}T_1 \\ \vdots \\ T_s \end{pmatrix}}{\to} \wtil{C}_0 \to \Z \to 0
\]
over $R$.
(Alternatively, this exact sequence is obtained from the Koszul complex for the regular sequence $T_1, \dots, T_s$.)
This implies that $\Fitt_R(N_s(T_1, \dots, T_s))$ is the Fitting ideal of the augmentation ideal of $R$.
Since $s \geq 1$, the augmentation ideal of $R$ is generically of rank one, so we obtain
$
\Fitt_R(N_s(T_1, \dots, T_s)) = 0,
$
as desired.

Finally we show the case where $1 \leq i < s$.
Since the components of the matrix $N_s(\tau_1, \dots, \tau_s)$ are either $0$ or one of $\tau_1, \dots, \tau_s$, the inclusion $\subset$ is clear.
In order to show the other inclusion, we use the induction on $s$.

For a while we fix an arbitrary $1 \leq l \leq s$.
Then, by permuting the rows and columns, the matrix $N_s(\tau_1, \dots, \tau_s)$ can be transformed into
\[
\begin{pmatrix}
N_{s - 1}(\tau_1, \dots, \check{\tau_l}, \dots, \tau_s)
& \\
\begin{array}{ccccc}
-\tau_l &&&& \\
& -\tau_l &&&\\ 
&& \ddots &&\\ 
&&& -\tau_l &\\ 
&&&& -\tau_l
\end{array}
& \begin{array}{c}
\tau_1 \\ \vdots \\ \check{\tau_l} \\ \vdots \\ \tau_s 
\end{array}
\end{pmatrix}.
\]
(The symbol $\check{(-)}$ means omitting that term.)
Here, the $x_l$ column is placed in the right-most, and the $x_1x_l, \dots, x_{l-1} x_{l}, x_{l} x_{l + 1}, \dots, x_l x_s$ rows are placed in the lower.
We also reversed the signs of some rows for readability as that does not matter at all.

This expression implies
\[
\Fitt_{i, \Z[G]}(N_s(\tau_1, \dots, \tau_s))
\supset (\tau_1, \dots, \check{\tau_l}, \dots, \tau_s) \Fitt_{i, \Z[G]}(N_{s-1}(\tau_1, \dots, \check{\tau_l}, \dots, \tau_s)).
\]
By the induction hypothesis (note that $1 \leq i \leq s - 1$), we have
\[
\Fitt_{i, \Z[G]}(N_s(\tau_1, \dots, \tau_s))
\supset (\tau_1, \dots, \check{\tau_l}, \dots, \tau_s) (\tau_1, \dots, \check{\tau_l}, \dots, \tau_s)^{s-1-i}
=(\tau_1, \dots, \check{\tau_l}, \dots, \tau_s)^{s-i}.
\]

Now we vary $l$ and then obtain
\[
\Fitt_{i, \Z[G]}(N_s(\tau_1, \dots, \tau_s))
\supset \sum_{l = 1}^s (\tau_1, \dots, \check{\tau_l}, \dots, \tau_s)^{s-i}
= (\tau_1, \dots, \tau_s)^{s-i},
\]
where the last equality follows from $s - i < s$.
This completes the proof of the proposition.
\end{proof}


Now we incorporate Propositions \ref{prop:04}, \ref{prop:02} and \ref{prop:03} to prove the following.

\begin{prop}\label{prop:FittB}
For $0 \leq i \leq s$, we define an ideal $J_i$ of $\Z[G]$ by
\[
J_i =
\begin{cases}
	(\nu_1 \cdots \nu_s) = (\nu_I) & (i = 0)\\
	\sum_{j = 0}^{s - i} Z_{i + j} \II_{I}^{j} 
	= Z_i + Z_{i + 1} \II_I + \cdots + Z_s \II_I^{s - i}
	& (1 \leq i \leq s).
\end{cases}
\]
Then we have
\[
\Fitt_{i, \Z[G]}(\II_I) = J_i.
\]
\end{prop}

\begin{proof}
By Propositions \ref{prop:04} and \ref{prop:02}, we have
\begin{align}
\Fitt_{i, \Z[G]}(\II_I)
&= \Fitt_{i, \Z[G]}(M)\\
&= \sum_{j=0}^{s - i} \sum_{\substack{a \subset \{1, 2, \dots, s\} \\ \# a = j}} \nu_{a_1} \cdots \nu_{a_j} \Fitt_{i, \Z[G]} (N_{s - j}(\tau_{a_{j + 1}}, \dots, \tau_{a_s})).
\end{align}

When $i = 0$, Proposition \ref{prop:03} implies
\[
\Fitt_{0, \Z[G]} (N_{s - j}(\tau_{a_{j + 1}}, \dots, \tau_{a_s})) =
\begin{cases}
(1) & (j = s)\\
0 & (0 \leq j < s).\\
\end{cases}
\]
Clearly, $j = s$ forces $a = \{1, 2, \dots, s\}$, so we obtain
\[
\Fitt_{0, \Z[G]}(\II_I) = (\nu_{1} \cdots \nu_{s}) = J_0.
\]

When $1 \leq i \leq s$, since $1 \leq i \leq s - j$ by the choice of $j$, Proposition \ref{prop:03} implies 
\[
\Fitt_{i, \Z[G]} (N_{s - j}(\tau_{a_{j + 1}}, \dots, \tau_{a_s}))
= (\tau_{a_{j + 1}}, \dots, \tau_{a_s})^{s-i-j}.
\]
Then we obtain
\[
\Fitt_{i, \Z[G]}(\II_I)
= \sum_{j=0}^{s - i} \sum_{\substack{a \subset \{1, 2, \dots, s\} \\ \# a = j}} \nu_{a_1} \cdots \nu_{a_j} (\tau_{a_{j+1}}, \dots, \tau_{a_s})^{s - i - j}.
\]
Using the relation $\nu_l \tau_l = 0$, for each $0 \leq j \leq s - i$, we have
\[
\sum_{\substack{a \subset \{1, 2, \dots, s\} \\ \# a = j}} \nu_{a_1} \cdots \nu_{a_j} (\tau_{a_{j+1}}, \dots, \tau_{a_s})^{s - i - j}
= \sum_{\substack{a \subset \{1, 2, \dots, s\} \\ \# a = j}} \nu_{a_1} \cdots \nu_{a_j} \II_I^{s - i - j}
= Z_{s - j }\II_I^{s - i - j}.
\]
These formulas imply $\Fitt_{i, \Z[G]}(\II_I) = J_i$.
\end{proof}

We are ready to prove Theorem \ref{Main theorem 2}.

\begin{proof}[Proof of Theorem \ref{Main theorem 2}]
By Propositions \ref{prop:01} and \ref{prop:FittB}, we have
\[
\SFitt{1}_{\Z[G]}(A) = \sum_{i = 0}^s \wtil{g}^{i-1} J_i.
\]
Then, noting $J_0 = (\nu_I)$, we can deduce
\[
h \SFitt{1}_{\Z[G]}(A) 
= \parenth{\nu_I, \parenth{1 - \frac{\nu_I}{\# I} \varphi^{-1} } \sum_{i = 1}^s \wtil{g}^{i-1} J_i}
\]
in the same way as in the proof of Proposition \ref{prop:Fitt-1}.
Then it is enough to show
\begin{equation}\label{eq:J_vs_Ji}
\JJ = \sum_{i = 1}^s \wtil{g}^{i-1} J_i.
\end{equation}

We claim that 
\begin{equation}\label{eq:J_1}
(\II_I, \# I) J_{i + 1} \subset J_i
\end{equation}
holds for $1 \leq i \leq s - 1$.
We first see
\[
\II_I J_{i + 1} 
= \II_I \sum_{j = 0}^{s - i - 1} Z_{i + 1 + j} \II_{I}^{j}
= \sum_{j = 1}^{s - i} Z_{i + j} \II_{I}^{j}
\subset J_i.
\]
We also have $\nu_I J_{i + 1} \subset (\nu_I) \subset J_0 \subset J_i$.
Since $(\II_I, \# I) = (\II_I, \nu_I)$ as an ideal, these show the claim \eqref{eq:J_1}.

Using \eqref{eq:J_1}, we next show
\begin{equation}\label{eq:J_2}
\sum_{i = 1}^s \wtil{g}^{i-1} J_i 
= \sum_{i = 1}^s (1 - \wtil{\varphi}^{-1})^{i-1} J_i.
\end{equation}
More generally we actually show
\[
\sum_{i = 1}^{s'} \wtil{g}^{i-1} J_i 
= \sum_{i = 1}^{s'} (1 - \wtil{\varphi}^{-1})^{i-1} J_i
\]
by induction on $s'$, for each $0 \leq s' \leq s$.
The case $s' = 0$ is trivial.
For $1 \leq s' \leq s$, we have
\begin{align}
\sum_{i = 1}^{s'} \wtil{g}^{i-1} J_i 
& = \wtil{g}^{{s'}-1} J_{s'}
	+ \sum_{i = 1}^{s' - 1} \wtil{g}^{i-1} J_i \\
& = \parenth{\sum_{i = 1}^{s'} (1 - \wtil{\varphi}^{-1})^{i - 1} (\# I)^{s' - i}} J_{s'}
	+ \sum_{i = 1}^{s' - 1} (1 - \wtil{\varphi}^{-1})^{i-1} J_i.
\end{align}
Here, the second equality follows from the induction hypothesis and expanding the power $\wtil{g}^{s' - 1}$.
By \eqref{eq:J_1}, for $1 \leq i \leq s' - 1$, we have 
$(\# I)^{s' - i} J_{s'}
\subset J_{i}.
$
Therefore, we obtain
\begin{align}
\sum_{i = 1}^{s'} \wtil{g}^{i-1} J_i 
& = (1 - \wtil{\varphi}^{-1})^{s' - 1} J_{s'}
	+ \sum_{i = 1}^{s' - 1} (1 - \wtil{\varphi}^{-1})^{i-1} J_i\\
& = \sum_{i = 1}^{s'} (1 - \wtil{\varphi}^{-1})^{i-1} J_i.
\end{align}
This completes the proof of \eqref{eq:J_2}.

The right hand side of \eqref{eq:J_2} can be computed as
\begin{align}
\sum_{i = 1}^s (1 - \wtil{\varphi}^{-1})^{i-1} J_i
& = \sum_{i = 1}^s \sum_{j=0}^{s-i} Z_{i+j} \II_I^j (1 - \wtil{\varphi}^{-1})^{i-1}\\
& = \sum_{k = 1}^s \sum_{j=0}^{k} Z_{k} \II_I^j (1 - \wtil{\varphi}^{-1})^{k - j -1}\\
& = \sum_{k = 1}^s Z_{k} \II_D^{k -1}
 = \JJ.
\end{align}
Here, the first equality follows from the definition of $J_i$, the second by putting $i + j = k$, the third by $\II_D = (\II_I, 1 - \wtil{\varphi}^{-1})$, and the final by the definition of $\JJ$.
Then, combining this with \eqref{eq:J_2}, we obtain the formula \eqref{eq:J_vs_Ji}.
This completes the proof of Theorem \ref{Main theorem 2}.
\end{proof}

\section{Stickelberger element and Fitting ideal}\label{sec:Stickel}

In this section, we prove Theorem \ref{Main theorem 3}.
As explained after the statement, we need to compare the ideals in the both sides of \eqref{eq:e11} for each $v$ before taking the product.
That task will be done in \S \ref{subsec:local}, and after that we complete the proof of Theorem \ref{Main theorem 3} in \S \ref{subsec:proof_3}.

In this section we fix an odd prime number $p$ and always work over $\Z_p$.

\subsection{Comparison of ideals}\label{subsec:local}

In this subsection, we again consider the general algebraic situation as in \S \ref{Intro2}.
Our task in this subsection is to compare the two fractional ideals
\[
\cA = h \SFitt{1}_{\Z_p[G]}(A \otimes \Z_p),
\qquad
\cB = \parenth{1 - \frac{\nu_{I}}{\# I} \varphi^{-1}}
\]
of $\Z_p[G]$.
In Lemma \ref{lem:e01} (resp. Lemma \ref{lem:e04}) below, we deal with the case where $D$ {\it is not} (resp. {\it is}) a $p$-group.
We will make use of the concrete description of $\cA$ in Theorem \ref{Main theorem 2}. 
As we always work over $\Z_p$ instead of $\Z$, by abuse of notation, in this subsection we simply write $\II_I$, $\II_D$, $Z_i$, and $\JJ$ for the extensions of those ideals from $\Z[G]$ to $\Z_p[G]$.
We have no afraid of confusion due to this.

%

Let $G'$ denote the maximal subgroup of $G$ whose order is prime to $p$.

\begin{lem}\label{lem:e01}
Let $\chi$ be a faithful character of $G'$.
Suppose that $D$ is not a $p$-group.
Then we have
\[
\cA^{\chi} = \cB^{\chi}
\]
as fractional ideals of $\Z_p[G]^{\chi}$.
\end{lem}

\begin{proof}
We write $D' = D \cap G'$ and $I' = I \cap G'$.
We first note that $\II_D^{\chi} = (1)$.
This is because $\chi$ is non-trivial on $D'$ by the assumptions.
Then we have $\JJ^{\chi} = (1)$ by Definition \ref{defn:J}, so Theorem \ref{Main theorem 2} implies
\[
\cA^{\chi}
= \parenth{\nu_I, \parenth{1 - \frac{\nu_I}{\# I} \varphi^{-1} }}^{\chi}.
\]

We have to show 
\[
\nu_{I}^{\chi} \in \parenth{1 - \frac{\nu_{I}}{\# I} \varphi^{-1}}^{\chi}.
\]
When $I'$ is non-trivial, then $\nu_I^{\chi} = 0$ as $\chi$ is non-trivial on $I'$, so this is obvious.
Let us suppose that $I'$ is trivial.
Since $\nu_{I}^2 = (\# I) \nu_{I}$, we have
\[
\nu_{I} \parenth{1 - \frac{\nu_{I}}{\# I} \varphi^{-1}} 
= \nu_{I} \parenth{1 - \varphi^{-1}}.
\]
The element $(1 - \wtil{\varphi}^{-1})^{\chi}$ of $\Z_p[G]^{\chi}$ is a unit 
since $\II_D = (\II_I, 1 - \wtil{\varphi}^{-1})$, $\II_D^{\chi} = (1)$, and $\II_I^{\chi} \subsetneqq (1)$.
This completes the proof.
\end{proof}

\begin{lem}\label{lem:e04}
Suppose that $I$ is non-trivial and that $D$ is a $p$-group.
Let $s = \rank_p(I)$ be the $p$-rank of $I$, that is, the number of minimal generators of $I$ (note that $s \geq 1$).
\begin{itemize}
\item[(1)]
We have
\[
\cA \supset
\II_{D}^{s - 1} \cB
\]
as fractional ideals of $\Z_p[G]$.

\item[(2)]
Let $\psi$ be a character of $G$ such that $\psi|_{G'}$ is faithful on $G'$ and that $\psi$ is non-trivial on $D$.
Then we have
\[
\psi (\cA)
= \psi \parenth{\II_{D}}^{s - 1} \psi (\cB)
\]
as ideals of $\OO_{\psi}$.
\end{itemize}
\end{lem}

\begin{proof}
We may take a decomposition \eqref{eq:I_decomp} of $I$ so that $s$ coincides with the $p$-rank of $I$ as the lemma, and then $I_l$ is non-trivial for each $1 \leq l \leq s$.

(1)
By Definition \ref{defn:J}, we have $\II_{D}^{s - 1} \subset \JJ$ (by the $i = s$ term as $Z_s = (1)$).
Then Theorem \ref{Main theorem 2} shows the claim (1).

(2)
We first show $\psi(\JJ) = \psi(\II_{D})^{s - 1}$.
By the claim (1) above, the inclusion $\psi(\JJ) \supset \psi(\II_{D})^{s - 1}$ holds.
For each $1 \leq l \leq s$, we observe $(\psi(\nu_l)) \subset (p)$ since $\psi(\nu_l)$ is either $0$ or $\# I_l$.
Moreover, we have $(p) \subset \psi(\II_{D})$ since $\psi$ is non-trivial on $D$ and we have $(p) \subset (1 - \zeta)$ if $\zeta$ is any non-trivial root of unity.
These observations imply $\psi(Z_i) \subset \psi(\II_D)^{s - i}$ for $1 \leq i \leq s$.
By the definition of $\JJ$, we then have $\psi(\JJ) \subset \psi(\II_D)^{s - 1}$ as claimed.

By Theorem \ref{Main theorem 2} and the above claim, we have
\[
\psi(\cA) = \parenth{\psi(\nu_I), \psi(\cB) \psi(\II_D)^{s - 1}}.
\]
We have to show $\psi(\nu_I) \in \psi(\cB) \psi(\II_D)^{s - 1}$.
This is obvious if $\psi$ is non-trivial on $I$.
If $\psi$ is trivial on $I$, we have
\[
\psi(\nu_I) = \# I \in (p^s) \subset \psi(\cB) \psi(\II_D)^{s - 1},
\]
where the last inclusion follows from $\psi(\cB) = \psi(\II_D) = (1 - \psi(\varphi)^{-1}) \supset (p)$.
This completes the proof of (2).
\end{proof}

\subsection{Proof of Theorem \ref{Main theorem 3}}\label{subsec:proof_3}

Now we consider the setup in \S \ref{Intro3}.
In particular, we fix an odd prime number $p$ and an odd character $\chi$ of $G'$.
Recall the $\chi$-component of the Stickelberger element $\theta_{K/k, T}^{\chi}$ defined as \eqref{eq:defn_theta}

\begin{lem}\label{lem:psi}
We have $\theta_{K/k, T}^{\chi} \neq 0$ if and only if there exists a character $\psi$ of $G$ such that $\psi|_{G'} = \chi$ and that $\psi$ is non-trivial on $D_v$ for any $v \in S_{\chi} \setminus S_{\infty}(k)$.
\end{lem}

\begin{proof}
By \eqref{eq:theta} and the fact that $\omega_T^{\chi}$ is a non-zero-divisor, 
we have $\theta_{K/k, T}^{\chi} \neq 0$ if and only if there exists a character $\psi$ of $G$ such that $\psi|_{G'} = \chi$ and, for every $v \in S_{\chi} \setminus S_{\infty}(k)$, we have $\psi \parenth{1 - \frac{\nu_{I_v}}{\# I_v} \varphi_v^{-1}} \neq 0$. 
The last condition is equivalent to that $\psi$ is non-trivial on $D_v$.
This proves the lemma.
\end{proof}

We begin the proof of Theorem \ref{Main theorem 3}.

\begin{proof}[Proof of Theorem \ref{Main theorem 3}]
As already remarked in the outline of the proof after Theorem \ref{Main theorem 3}, 
we may and do assume that $\chi$ is a faithful character of $G'$.
This is because we have $(\Cl_K^{T} \otimes \Z_p) \otimes_{\Z_p[G]} \Z_p[\Gal(K_{\chi}/k)] \simeq \Cl_{K_{\chi}}^{T} \otimes \Z_p$ as the degree of $K/K_{\chi}$ is prime to $p$.
Moreover, to simplify the notation, we write $S = S_{\chi} = S_{\ram}(K/k)$ and $S_{\fin} = S \setminus S_{\infty}(k)$.
Recall that, by Theorem \ref{Main theorem 1}, the condition (i) is equivalent to \eqref{eq:e11}.
As in \S \ref{subsec:local}, for each $v \in S_{\fin}$, we consider the fractional ideals of $\Z_p[G]$
\[
\cA_v = h_v \SFitt{1}_{\Z_p[G]}(A_v \otimes \Z_p),
\qquad
\cB_v = \parenth{1 - \frac{\nu_{I_v}}{\# I_v} \varphi_v^{-1}}.
\]

We first suppose (ii) and aim at showing (i).
The case where $\theta_{K/k, T}^{\chi} = 0$ is trivial, so we assume that, for any $v \in S_{\fin}$, either (a) or (b) holds.
Then we obtain $\cB_v^{\chi} \subset \cA_v^{\chi}$ for any $v \in S_{\fin}$, by applying Lemma \ref{lem:e01} (resp. Lemma \ref{lem:e04}(1)) if (a) (resp. (b)) holds.
Thus \eqref{eq:e11} holds.

We now prove that (i) implies (ii).
Suppose that both (i) and the negation of (ii) hold.
Since $\theta_{K/k, T}^{\chi} \neq 0$, we may take a character $\psi$ as in Lemma \ref{lem:psi}.
By applying $\psi$ to \eqref{eq:e11}, we obtain
\[
\prod_{v \in S_{\fin}} \psi (\cB_v)
\subset \prod_{v \in S_{\fin}} \psi (\cA_v).
\]
On the other hand, by Lemmas \ref{lem:e01} and \ref{lem:e04}(2), for each $v \in S_{\fin}$, we have
$\psi(\cA_v) \subset \psi(\cB_v)$.
Moreover, the inclusion is proper if and only if both the conditions (a) and (b) in (ii) are false.
Therefore, by the hypothesis that (ii) fails, we obtain
\[
\prod_{v \in S_{\fin}} \psi(\cA_v)
\subsetneqq \prod_{v \in S_{\fin}} \psi(\cB_v).
\]
Thus we get a contradiction.
This completes the proof of Theorem \ref{Main theorem 3}.
\end{proof}

\section*{Acknowledgments}

Both of the authors are sincerely grateful to Masato Kurihara for his continuous support during the research.
They also thank Cornelius Greither for extremely encouraging comments.
The second author is supported by JSPS KAKENHI Grant Number 19J00763.

{
\bibliographystyle{abbrv}
\bibliography{biblio}
}

\end{document}